\documentclass[12pt]{amsart} 
\usepackage{amssymb,amscd,amsmath}
\usepackage{verbatim}
\usepackage{graphics}
\usepackage{url}

\begin{document}


\numberwithin{equation}{section}

\newtheorem{theorem}[equation]{Theorem}
\newtheorem{lemma}[equation]{Lemma}
\newtheorem{conjecture}[equation]{Conjecture}
\newtheorem{proposition}[equation]{Proposition}
\newtheorem{question}[equation]{Question}
\newtheorem{corollary}[equation]{Corollary}
\newtheorem{cor}[equation]{Corollary}
\newtheorem*{theorem*}{Theorem}

\theoremstyle{definition}
\newtheorem*{definition}{Definition}
\newtheorem{example}[equation]{Example}

\theoremstyle{remark}
\newtheorem{remark}[equation]{Remark}
\newtheorem{remarks}[equation]{Remarks}
\newtheorem*{acknowledgement}{Acknowledgments}


\newenvironment{notation}[0]{%
  \begin{list}%
    {}%
    {\setlength{\itemindent}{0pt}
     \setlength{\labelwidth}{4\parindent}
     \setlength{\labelsep}{\parindent}
     \setlength{\leftmargin}{5\parindent}
     \setlength{\itemsep}{0pt}
     }%
   }%
  {\end{list}}

\newenvironment{parts}[0]{%
  \begin{list}{}%
    {\setlength{\itemindent}{0pt}
     \setlength{\labelwidth}{1.5\parindent}
     \setlength{\labelsep}{.5\parindent}
     \setlength{\leftmargin}{2\parindent}
     \setlength{\itemsep}{0pt}
     }%
   }%
  {\end{list}}
\newcommand{\Part}[1]{\item[\upshape#1]}

\renewcommand{\a}{\alpha}
\renewcommand{\b}{\beta}
\newcommand{\g}{\gamma}
\renewcommand{\d}{\delta}
\newcommand{\e}{\epsilon}
\newcommand{\f}{\phi}
\renewcommand{\l}{\lambda}
\renewcommand{\k}{\kappa}
\newcommand{\lhat}{\hat\lambda}
\newcommand{\m}{\mu}
\renewcommand{\o}{\omega}
\renewcommand{\r}{\rho}
\newcommand{\rbar}{{\bar\rho}}
\newcommand{\s}{\sigma}
\newcommand{\sbar}{{\bar\sigma}}
\renewcommand{\t}{\tau}
\newcommand{\z}{\zeta}

\newcommand{\D}{\Delta}

\newcommand{\gp}{{\mathfrak{p}}}
\newcommand{\gP}{{\mathfrak{P}}}
\newcommand{\gq}{{\mathfrak{q}}}
\newcommand{\gf}{{\mathfrak{f}}}

\newcommand{\Acal}{{\mathcal A}}
\newcommand{\Bcal}{{\mathcal B}}
\newcommand{\Ccal}{{\mathcal C}}
\newcommand{\Dcal}{{\mathcal D}}
\newcommand{\cD}{{\mathcal D}}
\newcommand{\Ecal}{{\mathcal E}}
\newcommand{\F}{{\mathcal F}}
\newcommand{\Fcal}{{\mathcal F}}
\newcommand{\cF}{{\mathcal F}}
\newcommand{\Gcal}{{\mathcal G}}
\newcommand{\Hcal}{{\mathcal H}}
\newcommand{\Ical}{{\mathcal I}}
\newcommand{\Jcal}{{\mathcal J}}
\newcommand{\Kcal}{{\mathcal K}}
\newcommand{\Lcal}{{\mathcal L}}
\newcommand{\cM}{{\mathcal M}}
\newcommand{\Mcal}{{\mathcal M}}
\newcommand{\Ncal}{{\mathcal N}}
\newcommand{\Ocal}{{\mathcal O}}
\newcommand{\Pcal}{{\mathcal P}}
\newcommand{\Qcal}{{\mathcal Q}}
\newcommand{\Rcal}{{\mathcal R}}
\newcommand{\Scal}{{\mathcal S}}
\newcommand{\Tcal}{{\mathcal T}}
\newcommand{\Ucal}{{\mathcal U}}
\newcommand{\Vcal}{{\mathcal V}}
\newcommand{\Wcal}{{\mathcal W}}
\newcommand{\Xcal}{{\mathcal X}}
\newcommand{\Ycal}{{\mathcal Y}}
\newcommand{\Zcal}{{\mathcal Z}}

\newcommand{\OO}{{\mathcal O}}    
\newcommand{\KK}{{\mathcal K}}    
\renewcommand{\O}{{\mathcal O}}   

\renewcommand{\AA}{\mathbb{A}}
\newcommand{\B}{\mathbb{B}}
\newcommand{\CC}{\mathbb{C}}
\newcommand{\EE}{\mathbb{E}}
\newcommand{\FF}{\mathbb{F}}
\newcommand{\GG}{\mathbb{G}}
\newcommand{\PP}{\mathbb{P}}
\newcommand{\NN}{\mathbb{N}}
\newcommand{\QQ}{\mathbb{Q}}
\newcommand{\RR}{\mathbb{R}}
\newcommand{\ZZ}{\mathbb{Z}}

\newcommand{\bfa}{{\mathbf a}}
\newcommand{\bfb}{{\mathbf b}}
\newcommand{\bfc}{{\mathbf c}}
\newcommand{\bfe}{{\mathbf e}}
\newcommand{\bff}{{\mathbf f}}
\newcommand{\bfg}{{\mathbf g}}
\newcommand{\bfp}{{\mathbf p}}
\newcommand{\bfr}{{\mathbf r}}
\newcommand{\bfs}{{\mathbf s}}
\newcommand{\bft}{{\mathbf t}}
\newcommand{\bfu}{{\mathbf u}}
\newcommand{\bfv}{{\mathbf v}}
\newcommand{\bfw}{{\mathbf w}}
\newcommand{\bfx}{{\mathbf x}}
\newcommand{\bfy}{{\mathbf y}}
\newcommand{\bfz}{{\mathbf z}}
\newcommand{\bfA}{{\mathbf A}}
\newcommand{\bfB}{{\mathbf B}}
\newcommand{\bfC}{{\mathbf C}}
\newcommand{\bfF}{{\mathbf F}}
\newcommand{\bfG}{{\mathbf G}}
\newcommand{\bfI}{{\mathbf I}}
\newcommand{\bfM}{{\mathbf M}}
\newcommand{\bfzero}{{\boldsymbol{0}}}
\newcommand{\bfmu}{{\boldsymbol\mu}}

\def\ta{{\tilde{a}}}
\def\tA{{\tilde{A}}}
\def\tb{{\tilde{b}}}
\def\tc{{\tilde{c}}}
\def\td{{\tilde{d}}}
\def\tf{{\tilde{f}}}
\def\tF{{\tilde{F}}}
\def\tg{{\tilde{g}}}
\def\tG{{\tilde{G}}}
\def\th{{\tilde{h}}}
\def\tK{{\tilde{K}}}
\def\tk{{\tilde{k}}}
\def\tz{{\tilde{z}}}
\def\tw{{\tilde{w}}}
\def\tphi{{\tilde{\varphi}}}
\def\tdelta{{\tilde{\delta}}}
\def\tbeta{{\tilde{\beta}}}
\def\talpha{{\tilde{\alpha}}}
\def\ttheta{{\tilde{\theta}}}
\def\tB{{\tilde{B}}}
\def\tR{{\tilde{R}}}
\def\tT{{\tilde{T}}}
\def\tE{{\tilde{E}}}
\def\tU{{\tilde{U}}}
\def\tGamma{{\tilde{\Gamma}}}

\newcommand{\ab}{{\textup{ab}}}
\newcommand{\Ahat}{{\hat A}}
\newcommand{\Aut}{\operatorname{Aut}}
\newcommand{\Cond}{{\mathfrak{N}}} 
\newcommand{\Disc}{\operatorname{Disc}}
\newcommand{\Div}{\operatorname{Div}}
\newcommand{\End}{\operatorname{End}}
\newcommand{\Frob}{\operatorname{Frob}}
\newcommand{\Fpbar}{{\overline{\FF}_p}}
\newcommand{\GK}{G_{\Kbar/K}}
\newcommand{\GL}{\operatorname{GL}}
\newcommand{\Gal}{\operatorname{Gal}}
\newcommand{\hhat}{{\hat h}}
\newcommand{\Image}{\operatorname{Image}}
\newcommand{\into}{\hookrightarrow}     
\newcommand{\Kbar}{{\bar K}}
\newcommand{\kvbar}{{{\bar k}_v}}
\newcommand{\Lbar}{{\bar L}}
\newcommand{\ellbar}{{\bar \ell}}
\newcommand{\Kvbar}{{{\bar K}_v}}
\newcommand{\MOD}[1]{~(\textup{mod}~#1)}
\newcommand{\Norm}{\operatorname{N}}
\newcommand{\notdivide}{\nmid}
\newcommand{\nr}{{\textup{nr}}}    
\newcommand{\ord}{\operatorname{ord}}
\newcommand{\Pic}{\operatorname{Pic}}
\newcommand{\Qbar}{{\bar{\QQ}}}
\newcommand{\Qpbar}{{\bar{\QQ}_p}}
\newcommand{\kbar}{{\bar{k}}}
\newcommand{\ksep}{{{k}^{\rm sep}}}
\newcommand{\QQbar}{{\bar{\QQ}}}
\newcommand{\rank}{\operatorname{rank}}
\newcommand{\res}{\operatornamewithlimits{res}}
\newcommand{\Resultant}{\operatorname{Res}}
\newcommand{\Res}{\operatorname{Res}}
\renewcommand{\setminus}{\smallsetminus}
\newcommand{\Spec}{\operatorname{Spec}}
\newcommand{\PGL}{\operatorname{PGL}}
\newcommand{\supp}{\operatorname{Supp}}
\newcommand{\tors}{{\textup{tors}}}
\newcommand{\val}{{\operatorname{val}}}
\newcommand{\<}{\langle}
\newcommand{\la}{{\langle}}
\renewcommand{\>}{\rangle}
\newcommand{\ra}{{\rangle}}
\newcommand{\Berk}{{\rm Berk}}
\newcommand{\BDV}{{\rm BDV}}
\newcommand{\Gauss}{{\rm Gauss}}
\newcommand{\Rat}{{\rm Rat}}
\newcommand{\RL}{{\rm RL}}
\newcommand{\HH}{{\mathbf H}}
\newcommand{\MM}{{\mathbf M}}
\newcommand{\Gm}{{\mathbf G}_m}
\newcommand{\Hhat}{{\hat{H}}}


\hyphenation{archi-me-dean}


\title[Preperiodic points and unlikely intersections]
{Preperiodic points and unlikely intersections}

\author{Matthew Baker}
\email{mbaker@math.gatech.edu}
\address{School of Mathematics,
          Georgia Institute of Technology, Atlanta GA 30332-0160, USA}

\author{Laura DeMarco}
\email{demarco@math.uic.edu}
\address{Department of Mathematics, Statistics, and Computer Science, 
University of Illinois at Chicago, Chicago, IL 60607-7045}

\date{August 2, 2010}

\begin{abstract}
In this article, we combine complex-analytic and arithmetic tools to study the preperiodic points of one-dimensional 
complex dynamical systems.  We show that for any fixed $a, b\in\CC$, and any integer $d \geq 2$, 
the set of $c\in\CC$ for which both $a$ and $b$ are preperiodic for $z^d+c$ is infinite if and only if $a^d = b^d$.  
This provides an affirmative answer to a question of Zannier, which itself arose from questions of Masser concerning simultaneous torsion sections on families of elliptic curves.
Using similar techniques, we prove that if rational functions 
$\varphi, \psi \in\CC(z)$ have infinitely many preperiodic points 
in common, then all of their preperiodic points coincide (and in particular 
the maps must have the same Julia set).  
This generalizes a theorem of Mimar, who established the same result assuming that $\varphi$ and $\psi$ are defined over $\Qbar$.
The main arithmetic ingredient in the proofs 
is an adelic equidistribution theorem for preperiodic points over number fields and function fields, 
with non-archimedean Berkovich spaces playing an essential role.
\end{abstract} 


\thanks{The research was supported by the National Science Foundation 
and the Sloan Foundation.  The authors
would like to thank Rob Benedetto, Xander Faber, Joe Silverman, and Umberto Zannier for helpful discussions, Daniel Connelly for  
computations related to Conjecture~\ref{BCConjecture}, and the
anonymous referees for useful suggestions.  We also thank Sarah Koch for generating the images in Figure \ref{Figure 1}.  
Finally, we thank AIM for sponsoring the January 2008 workshop that inspired the results of this paper.}

\maketitle

\thispagestyle{empty}




\section{Introduction}
\label{IntroSection}

\subsection{Statement of main results}

A complex number $a$ is {\em preperiodic} for 
a polynomial map $f \in \CC[z]$ if the forward orbit of $a$ under iteration
by $f$ is finite.  In this article, we examine preperiodic points for the unicritical polynomials, those of the form $z^d + c$ for a complex parameter $c$.  The main result of this article is the following.  

\begin{theorem}
\label{MainTheorem}
Let $d \geq 2$ be an integer, and fix $a,b \in \CC$. The set of parameters $c \in \CC$ such that both $a$ and $b$ are 
preperiodic for $z^d + c$ is infinite if and only if $a^d = b^d$.
\end{theorem}

One direction of Theorem~\ref{MainTheorem} follows easily from Montel's theorem: if $a^d = b^d$ then $a$ is preperiodic for $z^d + c$ if and only if $b$ is preperiodic, and the set of complex numbers $c \in \CC$ such that $a$ is preperiodic for $z^d + c$ is always infinite (see \S\ref{mandelbrot section}).  The reverse implication combines ideas from number theory and complex analysis, the main arithmetic ingredient  being an equidistribution theorem for points of {\em small height} 
with respect to an {\em adelic height function} (see \S\ref{EquidistSection} below for 
details).  
When $a$ and $b$ are algebraic, the equidistribution in question takes place over $\CC$, 
but in the transcendental case, we require an equidistribution theorem which takes place
on the Berkovich projective line\footnote{The Berkovich projective line $\PP^1_{\Berk,K}$ is a canonical compact,
Hausdorff, path-connected space containing $\PP^1(K)$ as a dense subspace.  It is for many applications the ``correct'' setting for non-archimedean potential theory and dynamics, see e.g. \cite{BakerAWS,BRBook,BenAWS,CantatCL,CL,FRL1,FRL2,Riv1,Thuillier}.} $\PP^1_{\Berk,K}$ over some complete and algebraically closed 
non-archimedean field $K$.  
The idea is to think of the field $k = \Qbar(a) \subset \CC$ as the function field of $\PP^1_{\Qbar}$, and to consider the distribution of the preperiodic parameters $c \in \kbar$ inside a collection of non-archimedean Berkovich analytic
spaces, one for each completion of $k$.

\medskip

Our method of proof also provides an analog of Theorem \ref{MainTheorem} in the dynamical plane, which we state in the more general context of rational functions.  For a rational function $\varphi \in \CC(z)$, let ${\rm Preper}(\varphi)$ denote its set of preperiodic points in $\PP^1(\CC)$.  

\begin{theorem}
\label{DynamicalTheorem}
Let $\varphi,\psi \in \CC(z)$ be rational functions of degrees at least 2.
Then ${\rm Preper}(\varphi) \cap {\rm Preper}(\psi)$ is infinite if and only if
${\rm Preper}(\varphi) = {\rm Preper}(\psi)$.
\end{theorem}

As a consequence, we have:

\begin{cor}
\label{DynamicalCor}
Let $\varphi, \psi \in\CC(z)$ be rational functions of degrees at least 2  with Julia sets $J(\varphi) \not= J(\psi)$.  Then the set of points which are preperiodic for both $\varphi$ and $\psi$ is finite.  
\end{cor}

When $\varphi$ and $\psi$ are defined over $\Qbar$, Corollary~\ref{DynamicalCor} is a special case of a result of Mimar
\cite{Mimar}.  In this setting, the proof of Theorem~\ref{DynamicalTheorem} shows that infinite intersection of ${\rm Preper}(\varphi)$ and ${\rm Preper}(\psi)$ implies equality of the canonical height functions $h_\varphi$ and $h_\psi$ on $\PP^1(\Qbar)$; this result is also proved in \cite{PetscheSzpiroTucker}.   In particular, the measures of maximal entropy for the complex dynamical systems $\varphi, \psi: \PP^1(\CC) \to \PP^1(\CC)$ must coincide.  For polynomials defined over a number field, we use a result of \cite{KawSilv} to strengthen the conclusion of Theorem \ref{DynamicalTheorem}; see Theorem \ref{RefinedDynamicalTheorem}.  

\medskip

The transcendental case of Theorem \ref{DynamicalTheorem}, which is handled using Berkovich spaces, is the main new ingredient.  We again deduce that canonical height functions $h_\varphi$ and $h_\psi$ are equal, only now they are defined over an appropriate function field.  In the absence of an archimedean absolute value, we are not able to deduce automatically that the measures of maximal entropy coincide, only that $\varphi$ and $\psi$ have the same set of preperiodic points.  

\medskip

Yuan and Zhang recently (and independently) proved a generalization of Theorem~\ref{DynamicalTheorem} to polarized algebraic dynamical systems of any dimension \cite{YuanZhang}.  Their proof also makes use of Berkovich spaces and equidistribution.

%

\medskip

The statements of Theorem \ref{MainTheorem} and Corollary \ref{DynamicalCor} are false if 
``are preperiodic for'' is replaced with ``are in the Julia set of''.    
For example, for any two points $a, b \in (-2, 2)$, there is an infinite family of polynomials $z^2 + c$, with 
$c \in\RR$ descending to $-2$, with Julia sets containing both $a$ and $b$ (in fact, containing a closed interval in 
$\RR$ increasing to $[-2,2]$ as $c \to -2$).  These same examples show that distinct Julia sets can have infinite 
intersections.  Because of such examples, it seems hard to prove results such as Theorems \ref{MainTheorem} and \ref{DynamicalTheorem} using complex analysis alone, without making use of some arithmetic information about preperiodic points. 

\medskip



In principle, it should be possible to extend the methods of the present paper to other one-parameter families of rational maps on $\PP^1(\CC)$ and to allow the points $a,b$ to depend algebraically on the parameter.   The special case we consider in Theorem \ref{MainTheorem} allows us a characterization by the simple condition $a^d \neq b^d$.  For more general one-parameter families $\{f_t\}$ and pairs $a(t), b(t)$, a precise condition characterizing finiteness is not as clear to formulate, and the complex analysis becomes more delicate.  However, most of the number-theoretic parts of our argument can easily be adapted to a more general setting.
It would also be interesting to study analogs of Theorem~\ref{MainTheorem} for, say, an $n$-dimensional family of rational maps in which $n+1$ points are required to be simultaneously preperiodic.

\subsection{Motivation and historical background}
The motivation for Theorem \ref{MainTheorem} came from a topic of discussion at the AIM workshop ``The uniform boundedness conjecture in arithmetic dynamics'' in Palo Alto in January 2008.  
By analogy with questions due to David Masser concerning 
simultaneous torsion sections on families of elliptic curves, Umberto Zannier asked:

\begin{question}
\label{MotivatingQuestion}
Is the set of complex numbers $c \in \CC$ such that $0$ and $1$ are both preperiodic for $z^2 + c$
finite?
\end{question}

Theorem \ref{MainTheorem} provides an affirmative answer to Question~\ref{MotivatingQuestion}, 
in analogy with the following recent theorem of Masser and Zannier:

\begin{theorem}  
\label{MZTheorem}  \cite{MasserZannier,MasserZannier2}
The set of complex numbers $\lambda \neq 0,1$ such that both $P_\lambda = (2,\sqrt{2(2-\lambda)})$
and $Q_\lambda = (3,\sqrt{6(3-\lambda)})$ have finite order on the Legendre elliptic curve
$E_\lambda$ defined by $Y^2 = X(X-1)(X-\lambda)$ is finite.
\end{theorem}

\noindent
As in Theorem~\ref{MainTheorem}, one shows fairly easily that there are infinitely many $\lambda$ such that either $P_\lambda$ or $Q_\lambda$
alone has finite order; however, in each case the set of such $\lambda$ is rather sparse (for example, it is countable), 
and imposing both torsion conditions at once makes the set of $\lambda$ finite. 
There is nothing special about the numbers $2$ and $3$; Masser and Zannier have announced that they can extend the main result of 
\cite{MasserZannier,MasserZannier2} to much more general linearly independent `sections' $P_{\lambda},Q_{\lambda}$.

\medskip

As it stands, the proof of Theorem~\ref{MainTheorem} is not effective (nor is the proof of Theorem~\ref{MZTheorem}), as we have no control over the number of parameters $c$ for which two points might be simultaneously preperiodic.  For example, it is easily checked that $0$ and $1$ are both preperiodic for $z^2 + c$ when $c \in \{ 0,-1,-2 \}$.  A straightforward computation with Mathematica shows that there are no other values of $c$ for which $f_c^{(\ell_0)}(0) = f_c^{(m_0)}(0)$ and $f_c^{(\ell_1)}(1) = f_c^{(m_1)}(1)$ with $0 \leq m_i < \ell_i < 15$ for $i=0,1$ and $\ell_0 + \ell_1 < 20$.   Letting $S_{0,1}$ denote the set of parameters $c$ for which both $0$ and $1$ are preperiodic, the following conjecture seems plausible:

\begin{conjecture}
\label{BCConjecture}
$S_{0,1} = \{ 0, -1, -2 \}$.
\end{conjecture}


\subsection{Brief overview of the proof of Theorem~\ref{MainTheorem}}
A key role in the proof of Theorem~\ref{MainTheorem} is played by certain generalizations of the
famous {\em Mandelbrot set}.
For $a \in \CC$, let $M_a$ denote the set of all $c \in \CC$ such that
$a$ stays bounded under iteration of $z^d + c$.  When $d = 2$ and $a = 0$, $M_{a}$
is just the usual Mandelbrot set.  See Figure \ref{Figure 1}.  We let $\mu_{a}$ denote the {\em equilibrium measure} on $M_{a}$ relative to $\infty$, in the sense of complex potential theory; by classical results, $\mu_a$ is a probability measure whose support is equal to $\partial M_{a}$.

\begin{figure}[!h]
\scalebox{.38}{\includegraphics{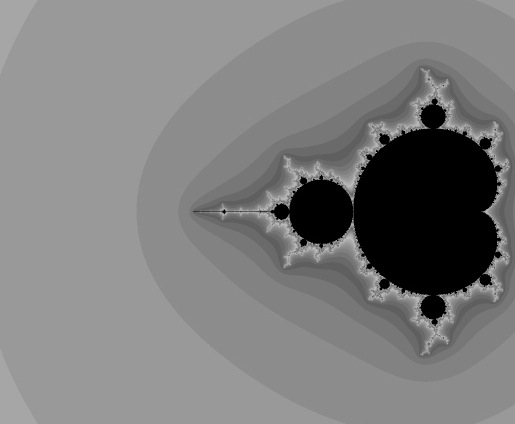}}\qquad\scalebox{.38}{\includegraphics{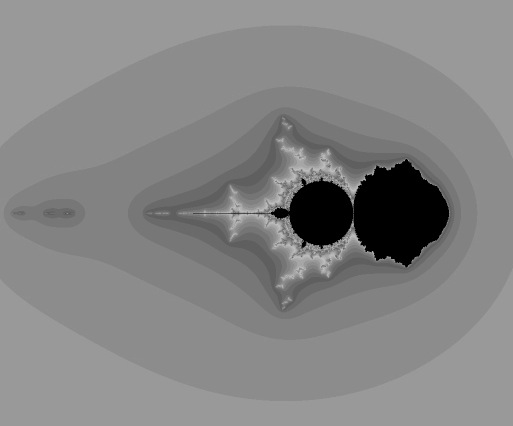}}
\caption{At left, the Mandelbrot set $M_0$, and at right, the set $M_1$ for the quadratic family $z^2+c$.  The images are centered at the same point; the shading indicates level sets of the Green's function.}
\label{Figure 1}
\end{figure}

\medskip

With this terminology in mind, an overview of the proof of Theorem~\ref{MainTheorem} is as follows.
Assume that there is an infinite sequence $c_1,c_2,\ldots$ of complex
numbers such that $a$ and $b$ are both preperiodic for $z^d + c_n$ for all $n$.

\medskip\noindent
{\bf Case 1:} $a,b$ are {\em algebraic} numbers.  In this case, all $c_n$'s must also be algebraic.
Let $\delta_n$ be the discrete probability measure on $\CC$ supported equally on the Galois conjugates of $c_n$.
Using the fact that $a$ is preperiodic for each $c_n$,
an arithmetic equidistribution theorem based on the product formula for number fields
shows that the measures $\delta_n$ converge weakly to the equilibrium measure $\mu_a$ for $M_a$ on $\PP^1(\CC)$.
By symmetry, the measures $\delta_n$ also converge weakly to $\mu_b$.
Thus $\mu_a = \mu_b$, which implies that $M_a = M_b$.
A complex-analytic argument using Green's functions and 
univalent function theory then shows that $a^d = b^d$.
(In the special case $a=0$ and $b=1$ corresponding to Question~\ref{MotivatingQuestion}, 
one can show directly that $M_0 \neq M_1$; for example, 
$i \in M_0$ but $i \not\in M_1$.  See Figure \ref{Figure 1}.)

\medskip\noindent
{\bf Case 2:} $a$ is transcendental.  In this case, one can show that $b$ is also transcendental, and that $a,b$, and all the $c_n$'s 
are defined over the algebraic closure $\kbar$ of $k = \Qbar(a)$ in $\CC$.  
The field $k$ is isomorphic to the field $\Qbar(T)$ of rational functions over the constant field $\Qbar$, 
and in particular $k$ has a (non-archimedean) product formula structure on it.  
An arithmetic equidistribution theorem based on the product formula for function fields,
together with the assumption that $a$ is preperiodic for each $c_n$,
shows that for every place $v$ of $k$, the $v$-adic analogue of the measures $\delta_n$ above
converge weakly on the Berkovich projective line $\PP^1_{\Berk,v}$ over $\CC_v$ 
to a probability measure whose support is the $v$-adic analogue $M_{a,v} \subseteq \PP^1_{\Berk,v}$ of $M_a$.
(Here $\CC_v$ denotes the completion of an algebraic closure of the $v$-adic completion $k_v$.)
By symmetry, it follows that $M_{a,v} = M_{b,v}$ for all places $v$ of $k$.
A theorem of Benedetto implies that for every $c \in \overline{k}$, and hence for every
complex number $c$, $a$ is preperiodic for $z^d + c$ if and only if $b$ is.  
We deduce using Montel's theorem that $M_a = M_b$, and finish the argument as in Case 1.  

\bigskip
\section{Potential theory background}

In this section we discuss some results from potential theory which are used in the rest of the paper.

\subsection{Complex potential theory}

Let $E$ be a compact subset of $\CC$.
The {\em logarithmic capacity} $\gamma(E)$ of $E$ relative to $\infty$ is $e^{-V(E)}$, where 
\begin{equation}
\label{CapacityDefinition}
-\log \gamma(E) = V(E) = \inf_{\nu} \iint_{E \times E} -\log |x-y| \, d\nu(x) \, d\nu(y). 
\end{equation}
The infimum in (\ref{CapacityDefinition}) is over all probability measures $\nu$ supported on $E$.
If $\gamma(E) > 0$ (equivalently, $V(E) < \infty$), 
then there is a unique probability measure $\mu_E$ which achieves
the infimum in (\ref{CapacityDefinition}), called the {\em equilibrium measure} for $E$.
The support of $\mu_E$ is contained in the ``outer boundary'' of $E$, i.e., in the 
boundary of the unbounded component $U_E$ of $\CC \backslash E$.

\medskip
If $\gamma(E) > 0$, the {\em Green's function}  $G_E$ is defined by
\[
G_E(z) = V(E) + \int_E \log |z - w| \; d\mu_E(w) ; 
\]
it is a nonnegative real-valued subharmonic function on $\CC$.  The following facts are well known; we include some proofs for lack of a convenient reference.

\begin{lemma}
\label{CpxGreenLemma}
Let $E$ be a compact subset of $\CC$ for which
$\gamma(E) = e^{-V(E)} > 0$, and let $U$ be the unbounded component of $\CC \backslash E$.  Then:
\begin{enumerate}
\item $G_E(z) = V(E) + \log |z| + o(1)$ for $|z|$ sufficiently large.
\item If $G : \CC \to \RR$ is a continuous subharmonic function which is harmonic on $U$, identically zero on $E$,
and such that $G(z) - \log^+|z|$ is bounded, then $G = G_E$.
\item If $G_E(z) = 0$ for all $z\in E$, then $G_E$ is continuous on $\CC$, 
$\supp \mu_E = \partial U$, and $G_E(z) > 0$ if and only if $z \in U$.
\end{enumerate}
\end{lemma}

\begin{proof}
Assertion (1) is \cite[Theorem 5.2.1]{Ransford}.

For (2), first note that $G_E$ is continuous at every point $q \in E$ where $G_E(q)=0$.
Indeed, $G_E$ is upper semicontinuous and bounded below by zero, so
\begin{equation}
\label{GreenContinuity}
0 \leq \liminf_{z \to q} G_E(z) \leq  \limsup_{z \to q} G_E(z) \leq G_E(q) = 0.
\end{equation}
By Frostman's Theorem (\cite[Theorem 3.3.4]{Ransford}), $G_E$ is 
identically zero on $\CC \backslash U$ outside a  set $e \subset \partial U$ of capacity $0$,
and hence the same is true for $f := G_E - G$.
Since $G_E$ is continuous on $\CC \backslash e$ 
and $G$ is continuous everywhere, $f$ is continuous outside $e$.
And by assumption, $f$ is harmonic and bounded on $U$.
By the Extended Maximum Principle \cite[Proposition 3.6.9]{Ransford}, 
we conclude that $f \equiv 0$ on $U$, and hence on $\CC \backslash e$.
Thus $G_E(z) = G(z)$ for all $z \in \CC \backslash e$.
Since $e$ has measure zero by \cite[Corollary 3.2.4]{Ransford},
the generalized Laplacians $\Delta(G_E)$ and $\Delta(G)$ coincide.   
Since $G_E$ and $G$ are both subharmonic on $\CC$,
it follows from Weyl's Lemma \cite[Lemma 3.7.10]{Ransford}
that $f$ is harmonic on all of $\CC$.
Since $f$ is also bounded, Liouville's Theorem \cite[Corollary 2.3.4]{Ransford}
implies that $f$ is identically zero.
This proves (2).

The continuity assertion in (3) follows from (\ref{GreenContinuity}), and the rest of (3) follows easily from the 
Maximum Principle.  
\end{proof}

\subsection{Non-archimedean potential theory}

In \cite{BRBook} (see also \cite{FRL2,Thuillier}), one finds non-archimedean Berkovich space analogs of 
various classical results from complex potential theory, including a
theory of Laplacians, harmonic functions, subharmonic functions, 
Green's functions, and capacities.  These results closely parallel
the classical theory over $\CC$.
For the reader's convenience, we give a quick summary in this section of the results from \cite{BRBook} which are used
in the present paper.\footnote{Amaury Thuillier has independently developed non-archimedean potential theory on $\PP^1_{\Berk,K}$
\cite{Thuillier}, and in fact his results are formulated in the context of arbitrary Berkovich curves,
and without assuming that the field $K$ is algebraically closed.
Also, Charles Favre and Juan Rivera-Letelier \cite{FRL1,FRL2} have independently developed most of the non-archimedean potential theory needed for the present applications to complex dynamics; their work relies heavily on potential theory for $\RR$-trees as developed in the book by Favre and Jonsson \cite{FJBook}.}  Although this theory is used heavily in the proofs of  Lemma~\ref{BerkGreenLemma}  and Theorem~\ref{AdelicHeightEquiTheorem}, the reader who wishes to accept these results as ``black boxes'' does not need a detailed understanding of non-archimedean potential theory in order to understand the proof of Theorem~\ref{MainTheorem} below.

\medskip

Let $K$ be an algebraically closed field which is complete with respect to some non-trivial absolute value $| \cdot |$.
The {\em Berkovich affine line} $\AA^1_{\Berk} = \AA^1_{\Berk,K}$ over $K$ is a 
locally compact, Hausdorff, path-connected space containing $K$
(with the given metric topology) as a dense subspace.  
As a topological space, $\AA^1_{\Berk,K}$ is the set
of all multiplicative seminorms $[ \cdot]_x : K[T] \to \RR$
on the polynomial ring $K[T]$ which extend the given absolute value on $K$, endowed with the weakest topology for which 
$x \mapsto [f]_x$ is continuous for all $f \in K[T]$.
The {\em Berkovich projective line} $\PP^1_{\Berk,K}$ can be identified with the 
one-point compactification of $\AA^1_{\Berk,K}$,
with the extra point denoted $\infty$.
It is a consequence of the Gelfand-Mazur theorem that if $K = \CC$, then $\AA^1_{\Berk,\CC}$ is homeomorphic
to $\CC$ (and $\PP^1_{\Berk}$ is homeomorphic to the Riemann sphere $\PP^1(\CC)$).
When $K$ is non-archimedean, however, there are infinitely many multiplicative seminorms $x \in \AA^1_{\Berk,K}$ which do 
not come from evaluation at a point of $K$; for example, the {\em Gauss point} $\zeta_{\Gauss} \in \AA^1_{\Berk,K}$
corresponds to the seminorm $[f]_{\zeta_{\Gauss}} := \sup_{z \in K, |z| \leq 1} |f(z)|$.

\medskip

For the rest of this section, we assume that the absolute value on $K$ is {\em non-archimedean} and non-trivial.
If $z \in \AA^1_{\Berk}$, we will sometimes write $|z|$ instead of
the more cumbersome $[T]_z$; the function $z \mapsto |z|$ is a 
natural extension of the absolute value on $K$ to $\AA^1_{\Berk}$.
Similarly, if $f \in K[T]$ we will sometimes write $|f(z)|$ instead of 
$[f(T)]_z$.

\medskip

There is a canonical extension of the fundamental potential kernel $-\log|x-y|$ to $\AA^1_{\Berk}$.
It can be defined as $-\log \delta(x,y)$,
where $\delta(x,y)$ (called the {\em Hsia kernel} in \cite{BRBook}) is defined as
\[
\delta(x,y) := \limsup_{\substack{z,w \in K \\ z \to x, w \to y}} |z - w|.
\]

Let $E$ be a compact subset of $\AA^1_{\Berk}$.
The {\em logarithmic capacity} $\gamma(E)$ of $E$ relative to $\infty$ is $e^{-V(E)}$, where 
\begin{equation}
\label{NonarchCapacityDefinition}
-\log \gamma(E) = V(E) = 
\inf_{\nu} \iint_{E \times E} -\log \delta(x,y) d\nu(x) \d\nu(y).
\end{equation}
The infimum in (\ref{NonarchCapacityDefinition}) is over all probability measures $\nu$ supported on $E$.
If $\gamma(E) > 0$ (equivalently, $V(E) < \infty$), 
there is again a unique probability measure $\mu_E$ which achieves
the infimum in (\ref{NonarchCapacityDefinition}), called the {\em equilibrium measure} for $E$
relative to $\infty$.
The support of $\mu_E$ is contained in the outer boundary of $E$ (the 
boundary of the unbounded component of $\AA^1_{\Berk} \backslash E$).

\medskip
If $\gamma(E) > 0$,
the {\em Green's function of $E$ relative to infinity} is defined by
\[
G_{E}(z) = V(E) + \int_E \log \delta(z,w) \; d\mu_E(w) ;
\]
it is a nonnegative real-valued subharmonic 
(in the sense of \cite[Chapter 8]{BRBook})
function on $\AA^1_{\Berk}$. 
For example, if $E = \cD(0,1)$ is the closed unit disc in $\AA^1_{\Berk}$, defined as
\[
\cD(0,1) = \{ x \in \AA^1_{\Berk} \; : \; |x| \leq 1 \},
\]
then
\[
G_{E}(z) = \log \max \{ |z|, 1 \}.
\]

The following is the non-archimedean counterpart of Lemma~\ref{CpxGreenLemma}:

\begin{lemma}
\label{BerkGreenLemma}
Let $E$ be a compact subset of $\AA^1_{\Berk}$ for which
$\gamma(E) = e^{-V(E)} > 0$, and let $U$ be the unbounded component of $\AA^1_{\Berk} \backslash E$.  Then:
\begin{enumerate}
\item $G_E(z) = V(E) + \log |z|$ for all $z \in \AA^1_{\Berk}$ with $|z|$ sufficiently large.
\item If $G : \AA^1_{\Berk} \to \RR$ is a continuous subharmonic function which is harmonic on $U$, identically zero on $E$,
and such that $G(z) - \log^+|z|$ is bounded, then $G = G_E$.
\item If $G_E(z) = 0$ for all $z\in E$, then $G_E$ is continuous on $\AA^1_{\Berk}$, 
$\supp \mu_E = \partial U$, and $G_E(z) > 0$ if and only if $z \in U$.
\end{enumerate}
\end{lemma}

\begin{proof}
Assertion (1) follows from \cite[Proposition 7.37(A7)]{BRBook}, and
(3) is \cite[Corollary 7.39]{BRBook}.

For (2), note that by \cite[Proposition 7.37(A4)]{BRBook}, $G_E$ is 
identically zero on $\AA^1_{\Berk} \backslash U$ outside a set $e \subset \partial U$ of capacity $0$,
and hence the same is true for $f := G_E - G$.
Since $G_E$ is continuous on $\AA^1_{\Berk} \backslash e$ by \cite[Proposition 7.37(A5)]{BRBook}
and $G$ is continuous everywhere, $f$ is continuous outside $e$.
And by assumption, $f$ is harmonic and bounded on $U$.
By the Strong Maximum Principle \cite[Proposition 7.17]{BRBook}, 
we conclude that $f \equiv 0$ on $U$.
Thus $G_E(z) = G(z)$ for all $z \in \AA^1_{\Berk} \backslash e$.

Note that $G_E$ is subharmonic on $\AA^1_{\Berk}$ by \cite[Example 8.9]{BRBook}
and $G$ is subharmonic on  $\AA^1_{\Berk}$ by assumption.
Since $e \subset \PP^1(K)$ by \cite[Example 6.3]{BRBook}, and the Laplacian of a function on $\PP^1_{\Berk}$
depends only on its restriction to $\PP^1_{\Berk} \backslash \PP^1(K)$ (see \cite[Remark 5.12]{BRBook}), 
we have $\Delta_{\AA^1_{\Berk}}(G_E) = \Delta_{\AA^1_{\Berk}}(G)$.   
Since $G_E$ and $G$ are both subharmonic on $\AA^1_{\Berk}$,
have the same Laplacian, and agree on $\AA^1_{\Berk} \backslash K$,
it follows from \cite[Corollary 8.37]{BRBook} that $G = G_E$ on $\AA^1_{\Berk}$.
\end{proof}

\subsection{Adelic equidistribution of small points}
\label{EquidistSection} 

In this section, we state the arithmetic equidistribution result needed for our proof of 
Theorem~\ref{MainTheorem}.  In order to state the result (Theorem~\ref{AdelicHeightEquiTheorem} below),
we first need some definitions.

\begin{definition} \label{PFF}
A {\em product formula field}
is a field $k$, together with the following extra data:
\begin{itemize}
\item[(1)] a set $\cM_k$ of non-trivial absolute values on $k$ (which we may assume to be pairwise inequivalent), and
\item[(2)] for each $v \in \cM_k$, an integer $N_v \geq 1$
\end{itemize}
such that
\begin{itemize}
\item[(3)] for each $\alpha \in k^\times$, we have $|\alpha|_v = 1$ for all but finitely many $v \in \cM_k$, and
\item[(4)] every $\alpha \in k^\times$ satisfies the {\em product formula}
\index{product formula}%
\begin{equation*} 
\prod_{v \in \cM_k} |\alpha|_v^{N_v} \ = \ 1.
\end{equation*}  
\end{itemize}
\end{definition}

The most important examples of product formula fields are number fields and
function fields of normal projective varieties (see \cite[\S{2.3}]{LangDG} or \cite[\S{1.4.6}]{BombieriGubler}).
It is known (see \cite[Chapter 12, Theorem 3]{Artin}) that a product formula field for which at least one
$v \in \cM_k$ is archimedean must be a number field.
If all $v \in \cM_k$ are non-archimedean, then we define the {\em constant field} $k_0$ of $k$ 
to be the set of all $\alpha \in k$ such that $|\alpha|_v \leq 1$ for all $v \in \cM_k$.
By the product formula, if $\alpha \in k_0$ is nonzero then in fact $|\alpha|_v = 1$ for all $v \in \cM_k$.

\begin{remark}
\label{fgfieldremark}
Any finitely generated extension $k$ of an algebraically closed field $k_0$ can be endowed 
with a product formula structure in such a way that the field of constants of $k$ is $k_0$ 
(cf. \cite[Lemma 1.4.10]{BombieriGubler}).  
Indeed, every such field $k$ can be (non-canonically) identified with the function field of a normal projective variety $X/k_0$. 
Choosing such an $X$, one endows the function field with a product formula structure 
in which the places $v \in \cM_k$ (which are all non-archimedean) correspond to prime divisors on $X$.
See \cite[\S{1.4}]{BombieriGubler} for further details.
\end{remark}


\medskip

Let $k$ be a product formula field and let 
$\kbar$ (resp. $\ksep$) denote a fixed algebraic closure (resp. separable closure) of $k$. 
For $v \in \cM_k$, let $k_v$ be the completion of $k$ at $v$, 
let $\kvbar$ be an algebraic closure of $k_v$, 
and let $\CC_v$ denote the completion of $\kvbar$.  
For each $v \in \cM_k$, we fix an embedding of $\kbar$ in $\CC_v$ extending the canonical embedding of $k$ in $k_v$,
and view this embedding as an identification.  
By the discussion above, if $v$ is archimedean then $\CC_v \cong \CC$.
For each $v \in \cM_k$, we let $\PP^1_{\Berk,v}$ denote the Berkovich projective line over $\CC_v$, 
which we take to mean $\PP^1(\CC)$ if $v$ is archimedean.

\medskip

A {\em compact Berkovich adelic set} (relative to $\infty$) is a set of the form 
\[
\EE \ = \ \prod_v E_v
\] 
where  $E_v$ is a nonempty compact subset of 
$\AA^1_{\Berk,v} = \PP^1_{\Berk, v} \backslash \{ \infty \}$ for each $v \in \cM_k$,
and where $E_v$ is the closed unit disc $\cD(0,1)$ in $\AA^1_{\Berk,v}$
for all but finitely many nonarchimedean $v \in \cM_k$.

\medskip
For each $v \in \cM_k$, let $\gamma(E_v)$ be the logarithmic capacity 
of $E_v$ relative to $\infty$; see (\ref{CapacityDefinition}) and (\ref{NonarchCapacityDefinition}).
The {\em logarithmic capacity} (relative to $\infty$) of a compact Berkovich adelic set $\EE$, 
denoted $\gamma(\EE)$, is 
\[
\gamma(\EE) \ = \ \prod_{v} \gamma(E_v)^{N_v} \ .
\]
We will assume throughout the rest of this section that 
$\gamma(\EE) \neq 0$, i.e., that $\gamma(E_v) > 0$ for all $v \in \cM_k$.

\medskip
For each $v \in \cM_k$, let $G_v : \AA^1_{\Berk,v} \to \RR$ be the Green's function for 
$E_v$ relative to $\infty$, i.e., $G_v(z) = G_{E_v}(z)$.
If $S \subset \ksep$ is any finite set invariant under $\Gal(\ksep / k)$,
we define the {\em height of $S$ relative to $\EE$}, denoted $h_{\EE}(S)$, by
\begin{equation} \label{HeightOfSDef2}
h_{\EE}(S) \ = \ \sum_{v \in \cM_k} N_v \left( 
\frac{1}{|S|} \sum_{z \in S} G_v(z) \right)  .
\end{equation}
By Galois-invariance, the sum $\sum_{z \in S} G_v(z)$
does not depend on our choice of an embedding of $\kbar$ into $\CC_v$.  

\medskip
If $z \in \ksep$, let $S_k(z) = \{ z_1,\ldots,z_n \}$ denote
the set of $\Gal(\ksep/k)$-conjugates of $z$ over $k$, where $n = [k(z):k]$.
We define a function $h_{\EE} : \ksep \to \RR_{\geq 0}$ by setting $h_{\EE}(z) = h_{\EE}(S_k(z))$.
If $E_v = \cD(0,1)$ for all $v \in \cM_k$, then $G_v(z) = \log_v^+|z|_v$ for all $v \in \cM_k$ and all 
$z \in \ksep$, and $h_{\EE}$ coincides with the {\em standard logarithmic Weil height} $h$.

\medskip

Finally, we let $\mu_v$ denote the equilibrium measure for $E_v$ relative to $\infty$.
We can now state the needed equidistribution result \cite[Theorem 7.52]{BRBook}:

\begin{theorem}
\label{AdelicHeightEquiTheorem}
Let $k$ be a product formula field 
and let $\EE$ be a compact Berkovich adelic set with $\gamma(\EE) = 1$.
Suppose $S_n$ is a sequence of $\Gal(\ksep/k)$-invariant finite subsets of $\ksep$ with $|S_n| \to \infty$
and $h_{\EE}(S_n) \to 0$. 
Fix $v \in \cM_k$, and for each $n$ let $\delta_n$ be the discrete probability measure on
$\PP^1_{\Berk,v}$ supported equally on the elements of $S_n$.
Then the sequence of measures $\{\delta_n\}$ converges weakly
to $\mu_v$ on $\PP^1_{\Berk,v}$.
\end{theorem}

\begin{remark}
When $k$ is a number field, a slightly weaker version of 
Theorem~\ref{AdelicHeightEquiTheorem} is proved in 
\cite{BREQUI} and a closely related result (which also generalizes
Theorem~\ref{RationalEquiThm} below) is proved in \cite{FRL2}.
\end{remark}

For concreteness, we explicitly state the special case of Theorem~\ref{AdelicHeightEquiTheorem}
in which $k$ is a number field and $S_n$ is the $\Gal(\kbar/k)$-orbit of a point $z_n \in \kbar$
(note in this case that $h_{\EE}(z_n)\to 0$ implies $|S_n|\to\infty$ by Northcott's theorem):

\begin{corollary}
\label{BREQUIcor2}
Let $k$ be a number field, and let $\EE$ be a compact Berkovich adelic set with $\gamma(\EE) = 1$.
Suppose $\{z_n\}$ is a sequence of distinct points of
$\kbar$ with $h_{\EE}(z_n) \to 0$. 
Fix a place $v$ of $k$,
and for each $n$ let $\delta_n$ be the discrete probability measure on
$\PP^1_{\Berk,v}$ supported equally on the $\Gal(\kbar/k)$-conjugates of $z_n$.  
Then the sequence of measures $\{\delta_n\}$ converges weakly
to $\mu_v$ on $\PP^1_{\Berk,v}$.
\end{corollary}

\begin{remark}
When $k = \QQ$ and $\EE$ is the {\em trivial} Berkovich adelic set 
(i.e., $E_v$ is the $v$-adic unit disc for all $v$), 
Corollary~\ref{BREQUIcor2} is Bilu's equidistribution theorem \cite{Bilu} for $v$ archimedean, and it is Chambert-Loir's generalization of Bilu's theorem \cite{CL} for $v$ non-archimedean.
\end{remark}

\begin{remark}
If $k$ is a number field and $\gamma(\EE)<1$, there are only finitely many $z \in \kbar$ with $h_{\EE}(z) = 0$; this follows from
the adelic version of the Fekete-Szeg{\"o} theorem proved in
\cite[Theorem 6.28]{BRBook}.  This observation helps explain the
role played by the condition $\gamma(\EE) = 1$ in
Theorem~\ref{AdelicHeightEquiTheorem} and Corollary~\ref{BREQUIcor2}.
\end{remark}

\bigskip
\section{Generalized Mandelbrot sets}
\label{mandelbrot section}

Let $K$ be an algebraically closed field which is complete with respect to a nontrivial 
(archimedean or nonarchimedean) absolute value.
Fix an integer $d \geq 2$, and for $c \in K$ let $f_c(z) = z^d + c$.  
We denote by $f^{(n)}_c$ the $n^{\rm th}$ iterate of $f_c$.  In this section, we introduce a family of generalized Mandelbrot sets, defined as the set of parameters $c$ for which a given point $z=a$ remains bounded under iteration.  

\subsection{The archimedean case}
\label{complexM}
If $K = \CC$, define the {\em generalized Mandelbrot set} $M_a$ for $a \in \CC$ by
\begin{equation}
\label{ComplexMdef}
M_{a} = \left\{ c \in \CC \; : \; \sup_n |f^{(n)}_c(a)| <  \infty \right\}.
\end{equation}
When $d=2$ and $a = 0$, $M_a$ is the usual Mandelbrot set.  It is clear that every parameter $c \in K$ for which $a$ is preperiodic for $z^d + c$ is contained in $M_{a}$.  See Figure \ref{Figure 1}.  

\medskip

We need some basic potential-theoretic properties of $M_{a}$.  The proofs follow the same reasoning as for the Mandelbrot set, but we provide some details for the reader's convenience.  Recall that for each fixed $f_c$, the Green's function for the filled Julia set 
	$$K_c = \left\{z: \sup_n |f_c^{(n)}(z)| < \infty \right\}$$
of $f_c$ is given by the  {\em escape rate}
	$$G_c(z) = \lim_{n\to\infty} \frac{1}{d^n} \log^+ |f_c^{(n)}(z)|.$$
These escape-rate functions are continuous in both $c$ and $z$, and $G_c(z) = 0$ if and only if $z \in K_c$.  In fact, as a locally uniform limit of plurisubharmonic functions, the function
	$$(c,z) \mapsto G_c(z)$$
is plurisubharmonic on $\CC\times\CC$.  For each fixed $c$ there is an analytic homeomorphism $\phi_c$, defining the {\em B\"ottcher coordinate} $w = \phi_c(z)$ near $\infty$, which satisfies $\phi_c(f_c(z)) = (\phi_c(z))^d$ and $G_c(z) = \log|\phi_c(z)|$.  The map $\phi_c$ sends the domain 
	$$V_c := \{z\in\CC: G_c(z) > G_c(0)\}$$
conformally to the punctured disk $\{w \in\CC: |w| > e^{G_c(0)}\}$.  The B\"ottcher coordinate is uniquely determined if we require that $\phi_c$ has derivative 1 at infinity.   See for example \cite{Douady:Hubbard:etude}, 
\cite[\S{2.4}]{Carleson:Gamelin}, and \cite{Branner:Hubbard:1}.  

\begin{lemma}  \label{Glemma}
For each fixed $a\in\CC$, we have $G_c(a^d+c) > G_c(0)$ for all $c$ sufficiently large.  Consequently, the value $f_c(a)$ lies in the domain $V_c$ of the conjugating isomorphism $\phi_c$.  
\end{lemma}

\proof
The proof relies on a standard distortion theorem for univalent functions 
\cite[Theorem~14.14]{Rudin:realandcomplex}:  for any holomorphic function
	$$h(z) = z + \sum_{n\geq 2} a_n z^n$$
which is univalent on a disk of radius $1/r$, we have $|a_2| \leq 2 r$.  The conjugating isomorphism $\phi_c$ satisfies $\phi_c(z) = z + O(1/z)$ for $z$ near infinity.  Let $U_R$ denote the domain $\{|z| > R\}$ in the complex plane.  Setting $R_c = e^{G_c(0)}$, the inverse function $\phi_c^{-1} : U_{R_c} \to V_c$ is univalent on $U_{R_c}$, and it also has expansion $\phi_c^{-1}(z) = z+O(1/z)$.  For any $w\not\in V_c$, consider the univalent function 
	$$h_w(z) = \frac{1}{\phi_c^{-1}(1/z) - w} = z + w z^2 + \cdots$$
We conclude that $|w| \leq 2R_c$, and therefore $V_c \supseteq U_{2R_c}$.  This argument appears in \cite[\S3]{Branner:Hubbard:1}.

In particular, since $G_c(w) = G_c(0)/d$ and thus $w \not\in V_c$ for every $w$ with $f_c(w) = 0$,
the critical point $z=0$ and all of its preimages $(-c)^{1/d}$ must lie in the closed disk of radius $2 R_c$.  Thus 
$|c| \leq 2^dR_c^d$.  This implies that $R_c \to \infty$ as $c\to\infty$.   

Note that $|\phi_c(c)| = R_c^d$.  When $c$ is large enough so that $R_c^d/2 > 2R_c$, we apply the same distortion estimate to conclude that $\phi_c(U_{R_c^d/2}) \supseteq U_{R_c^d}$, so $|c| \geq R_c^d/2$.  It follows that for any fixed $a$, since $R_c\to\infty$ with $c$, we have $|a^d + c| \geq R_c^d/2 - |a|^d > 2R_c$ for all sufficiently large $c$.  Thus 
$f_c(a) = a^d+c$ lies in $V_c$ and has escape rate $G_c(a^d+c) > G_c(0)$.  
\qed

\begin{proposition}  \label{M}
For each $a\in\CC$, the generalized Mandelbrot set $M_a$ satisfies:
\begin{enumerate}
\item		$M_a$ is a compact and full subset of $\CC$;
\item		the function $G_a(c) := G_c(a^d+c)$ defines the Green's function for $M_a$ and satisfies $G_a(c) = 0$ for all $c\in M_a$;
\item		the function $\Phi_a(c) := \phi_c(a^d+c)$ defines a conformal isomorphism in a neighborhood of infinity, and it is uniquely determined by the conditions $G_a(c) = \log|\Phi_a(c)|$ and $\Phi_a'(\infty) = 1$;
\item		the logarithmic capacity is $\gamma(M_a) = 1$; and
\item		the support of the equilibrium measure $\mu_{M_a}$ on $M_a$ is equal to the boundary $\partial M_a$.
\end{enumerate}
\end{proposition}

\proof
The set $M_a$ is closed because $M_a = \{c: G_c(a) = 0\}$ and $(c,z)\mapsto G_c(z)$ is continuous.  It is bounded by Lemma \ref{Glemma}:   for all sufficiently large $c$, the escape rate of $a$ is positive and therefore $f_c^{(n)}(a) \to \infty$.   The maximum principle applied to the subharmonic function $c\mapsto G_c(a)$ implies that $M_a$ is full (meaning that its complement is connected), completing the proof of statement (1).  

The conjugating isomorphisms $\phi_c$ satisfy 
	$$\phi_c(z) = z \prod_{n=0}^\infty \left( 1 + \frac{c}{(f^{(n)}_c(z))^d} \right)^{1/d^{n+1}} $$
on their domains $\{z: G_c(z)>G_c(0)\}$; see \cite{Douady:Hubbard:etude} or \cite[\S{VIII.3}]{Carleson:Gamelin}.  By Lemma \ref{Glemma}, the function $\Phi_a(c)/c = \phi_c(a^d + c)/c$ can be expressed by this infinite product for $c$ near infinity.  The terms in the infinite product each tend to 1 as $c\to \infty$, so (setting $\Phi_a(\infty) = \infty$) we conclude $\Phi_a'(\infty) = 1$.  In particular, $\Phi_a$ defines a conformal isomorphism in a neighborhood of infinity.  

The Green's function $G_c$ for the filled Julia set $K_c$ satisfies $G_c(z) = \log|\phi_c(z)|$ where defined.  The function $G_a(c) = G_c(a^d+c)$ therefore satisfies $G_a(c) = \log|\Phi_a(c)| = \log|c + O(1)| = \log|c| +  o(1)$ for all $c$ large.  Furthermore, $G_a$ is harmonic on $\CC\setminus M_a$, as a locally uniform limit of the harmonic functions $c \mapsto G_n(c) = d^{-n} \log |f_c^{(n)}(a^d + c)|$, and $G_a(c) = 0$ if and only if $c\in M_a$; we conclude that $G_a$ is the Green's function for $M_a$.  The conditions stated in (3) clearly determine $\Phi_a$ uniquely near $\infty$.  Statement (4) follows because $G_a(c) = \log|c| + o(1)$ near infinity.  

Finally, statement (5) follows from Lemma~\ref{CpxGreenLemma}, because $M_a$ is full.  
\qed

\medskip
Fix a degree $d\geq 2$.  For each $a \in \CC$, define
	$${\rm Preper}(a) := \{ c \in \CC \; : \; a \textrm{ is preperiodic for } z^d+c \}.$$
Combining Proposition \ref{M} with Montel's theorem, we obtain:  

\begin{theorem}  \label{Recovery}
For each degree $d\geq 2$ and any $a,b\in\CC$, the following are equivalent:
\begin{enumerate}
\item  $M_a = M_b$
\item $\mu_{M_a} = \mu_{M_b}$
\item  $a^d = b^d$
\item  ${\rm Preper}(a) = {\rm Preper}(b)$.
\end{enumerate}
\end{theorem}

\proof
First suppose that $a^d = b^d$.  Then for every $c$, we have $f_c(a) = a^d + c = b^d + c = f_c(b)$, so $a$ is preperiodic for $f_c$ if and only if $b$ is preperiodic for $f_c$.  Thus (3) implies (4).  

Now assume (4) and consider the sequence of functions $g_n(c) := f_c^{(n)}(a)$.  This sequence forms a normal family except on the boundary $\partial M_a$.  Consider the holomorphic functions $h_1(c) = a$ and $h_2(c) = a^d+c$.  First note that $a^d + c = a$ implies that $a$ is a fixed point for $f_c$, so $c = a - a^d \in {\rm Preper}(a) \subset M_a$.  Now fix an open set $U$ intersecting $\partial M_a$ which does not contain the parameter $c = a - a^d$, so the images of $h_1$ and $h_2$ are distinct throughout $U$.   Then by Montel's theorem, the failure of normality of $\{g_n\}$ implies that the image of $g_n$ must coincide with that of $h_1$ or $h_2$ for some $n$ and some $c\in U$.  In particular, there is an iterate $n$ so that either $f^{(n)}_c(a) = a$ or $f^{(n)}_c(a) = f_c(a)$; we conclude that the set $U$ must intersect ${\rm Preper}(a)$.   Consequently, the boundary $\partial M_a$ is contained in the closure of ${\rm Preper}(a)$.  
As ${\rm Preper}(a) \subset M_a$ by definition and $M_a$ is a full set by Proposition \ref{M} (1), the assumption 
${\rm Preper}(a) = {\rm Preper}(b)$ implies that $M_a = M_b$, i.e., (4) implies (1).

Assume that $M_a = M_b$.  Clearly the equilibrium measures coincide, so $\mu_{M_a} = \mu_{M_b}$, giving (1) $\implies$ (2).  Conversely, the support of $\mu_{M_a}$ is equal to the boundary $\partial M_a$, so again using the fact that $M_a$ is a full set we conclude that (2) implies (1).  

Finally, if $M_a = M_b$, then by Proposition \ref{M} (3) the uniformizing maps $\Phi_a$ and $\Phi_b$ coincide on a neighborhood of infinity.  In other words, for all large $c$, we have $\phi_c(a^d + c) = \phi_c(b^d + c)$.  The conjugating isomorphisms $\phi_c$ are injective, so we conclude that $a^d + c = b^d +c$.  This shows that (1) implies (3), completing the proof.  
\qed

\medskip
The following simple statement is used for one implication of Theorem \ref{MainTheorem}.  
 
\begin{lemma}
\label{MontelLemma}
For each $a \in \CC$, the set ${\rm Preper}(a)$ is infinite.
\end{lemma}

\proof
From Proposition \ref{M}, the set $M_a$ has capacity 1, so its boundary cannot be a finite set.  From the proof of Theorem \ref{Recovery},  the boundary $\partial M_a$ is contained in the closure of ${\rm Preper}(a)$, so the set ${\rm Preper}(a)$ must contain infinitely many points.
\qed

\medskip
Note that if $a \in \Qbar$, then the set ${\rm Preper}(a)$ is a subset of $\Qbar$ with
bounded Weil height (since $h_{\MM_a}(c) = 0$ for all $c \in {\rm Preper}(a)$ and the difference 
$h - h_{\MM_a}$ is bounded).  It is thus a rather ``sparse'' set (compare with the discussion following
Theorem~\ref{MZTheorem} above).

\subsection{The non-archimedean case}
If $K$ is a non-archimedean field, one can define $M_a \subset \AA^1_{\Berk,K}$ similarly and prove basic potential-theoretic statements about $M_a$.  

Let $g_n(T) = f^{(n)}_T(a)$; this is a monic polynomial in $T$ of degree $d^{n-1}$ which depends on $a \in K$.
Define
\begin{equation}
\label{BerkovichMdef}
M_{a} := \left\{ c \in \AA^1_{\Berk,K} \; : \; \sup_n \, |g_n(c)| < \infty \right\},
\end{equation}
where $[ \cdot ]_c$ is the multiplicative seminorm on $K[T]$
corresponding to $c \in \AA^1_{\Berk,K}$
and $|g_n(c)|$ is shorthand for $[g_n(T)]_c$.
(Note that for $c \in K$, we have $[g_n(T)]_c = |g_n(c)| = |f^{(n)}_c(a)|$.)

\begin{proposition}
\label{capgreenprop}
For each $a\in K$, 
\begin{enumerate}
\item the boundary of $M_a$ coincides with the outer boundary in  $\AA^1_{\Berk,K}$, and it is equal to the support of $\mu_{M_a}$;
\item the logarithmic capacity $\gamma(M_a)$ is equal to $1$; and
\item the Green's function for $M_a$ relative to $\infty$ is 0 at all points of $M_a$.
\end{enumerate}
\end{proposition}

\proof
Fix $a \in K$, and for $c \in \AA^1_{\Berk}$ define
\begin{equation}
\label{nonarchGadef}
G_a(c) := \lim_{n \to \infty} \frac{1}{d^n} \log^+ |g_{n+1}(c)|.
\end{equation}
Note that for $c \in K$, we have
\[
G_a(c) = \lim_{n \to \infty} \frac{1}{d^n} \log^+ |f_c^{(n+1)}(a)|
= \lim_{n \to \infty} \frac{1}{d^n} \log^+ |f_c^{(n)}(a^d+c)|,
\]
which is the same formula we used to define $G_a$ over $\CC$.

\medskip

The limit in (\ref{nonarchGadef}) exists for all $c \in \AA^1_{\Berk}$:
if $c \in M_a$, then the limit is zero, while if $c \not\in M_a$,
then the sequence
$\frac{1}{d^n} \log^+ |g_{(n+1)}(c)|$
is eventually constant.
Indeed, if $c \not\in M_a$ then the sequence
$|g_n(c)|$ is unbounded  so there exists $N = N(c)$ such that 
$|g_n(c)| > \max \{ 1, |c|^{\frac{1}{d}} \}$ for $n \geq N$.
By the ultrametric inequality, for $c \in K$ and $n \geq N$ we have
\[
|g_{n+1}(c)| = |g_n(c)^d + c| = |g_n(c)|^d > 1,
\]
and since $K$ is dense in $\AA^1_{\Berk}$ the equality
$|g_{n+1}(c)| = |g_n(c)|^d$ holds more generally for all $c \in \AA^1_{\Berk} \backslash M_a$.  Thus the sequence $\frac{1}{d^n} \log^+ |g_{(n+1)}(c)|$ is
constant for $n \geq N$.
\medskip

{\bf Claim 1:} $\frac{1}{d^n} \log^+ |g_{n+1}(c)|$ converges to $G_a(c)$ uniformly on compact subsets of $\AA^1_{\Berk}$
(as functions of $c$).

\medskip

To see this, one can employ essentially the same argument as in the archimedean case
(\cite[Proposition 1.2]{Branner:Hubbard:1}; compare with \cite[\S{10.1}]{BRBook}). Briefly, fix a compact set $E \subset \AA^1_{\Berk}$.
Then there is a constant $C>0$, depending only on $E$, such that 
$[T + z^d]_c = |z|^d$ for $c \in E$ and $|z| \geq C$.
Thus there are constants $C_1,C_2 > 0$ (depending only on $E$) 
such that for $z \in K$ and $c \in E$, 
\[
C_1 \max (1,|z|^d) \leq \max(1,[T+z^d]_c) \leq C_2 \max (1,|z|^d).
\]
Taking logarithms, iterating, and multiplying by $d$ shows that there is a constant $C' > 0$ (depending only on $E$) such that
for each fixed $a \in K$,
\[
\left| \frac{1}{d^{n}} \log^+|g_{n+1}(c)| -  \frac{1}{d^{n-1}} \log^+|g_{n}(c)| \right| \leq \frac{C'}{d^n}.
\]
A telescoping series argument now gives the desired uniform convergence on $E$:
\[
\begin{aligned}
\left| G_a(c) - \frac{1}{d^{n-1}} \log^+ |g_n(c)| \right|
&= \left| \sum_{m \geq n} \left( 
\frac{1}{d^m} \log^+ |g_{m+1}(c)| - \frac{1}{d^{m-1}} \log^+ |g_{m}(c)| 
\right) \right| \\
&\leq  \sum_{m \geq n} \frac{C'}{d^m} = \frac{C'}{d^n - d^{n-1}}, \\
\end{aligned}
\]
proving Claim 1.

\medskip

{\bf Claim 2:} $G_a$ is the Green's function for $M_a$ relative to $\infty$.

\medskip

Indeed, $G_a$ is harmonic on $U_a := \AA^1_{\Berk} \backslash M_a$ by 
\cite[Example 7.5]{BRBook} and \cite[Proposition 7.31]{BRBook},
since on $U_a$ the function $G_a$ is the limit of the harmonic
functions $\frac{1}{d^n} \log |g_{n+1}(c)|$.
Moreover, since the sequence of continuous subharmonic functions $\frac{1}{d^n} \log^+ |g_{n+1}(c)|$ converges
uniformly to $G_a$ on compact subsets of $\AA^1_{\Berk}$,
it follows from \cite[Proposition 8.26(C)]{BRBook} that 
$G_a$ is continuous and subharmonic on $\AA^1_{\Berk}$.
In addition, $G_a$ is zero on $M_a$, and for $|c| > \max\{ 1,|a|^d \}$ we have
$G_a(c) = \log^+|c|$.
Claim 2 therefore follows from part (2) of Lemma~\ref{BerkGreenLemma}.

Assertion (3) is now immediate, and assertions (1) and (2) follow from parts (3) and (1) of
Lemma~\ref{BerkGreenLemma}, respectively.
\qed

\medskip
\subsection{Global generalized Mandelbrot sets}
\label{section:GGMS}

Let $k$ be a product formula field, and fix $a \in k$.  For each $v \in \cM_k$, define $M_{a,v} \subseteq \AA^1_{\Berk,\CC_v}$ following the local recipes above. Recall that $\AA^1_{\Berk,\CC_v} = \CC$ if $v$ is archimedean.  Define a compact Berkovich adelic set $\MM_a$ by 
	$$\MM_a := \{ M_{a,v} \},$$ 
observing that $M_{a,v} = \cD(0,1)$ whenever $|a|_v \leq 1$.  Propositions \ref{M} (4) and \ref{capgreenprop} (2) imply that the global capacity $\gamma(\MM_a)$ is equal to $1$.
Moreover, for each $v \in \cM_k$ the local Green's function 
$G_{M_{a,v}} : \AA^1_{\Berk,v} \to \RR_{\geq 0}$ 
is continuous, with $G_{M_{a,v}}(z) = 0$ if and only if  $z \in M_{a,v}$.

If $S \subset \kbar$ is any finite set invariant under $\Gal(\kbar / k)$,
then following (\ref{HeightOfSDef2}) the height of $S$ relative to $\MM_a$ is given by 
\begin{equation} \label{HeightOfSDef3}
h_{\MM_a}(S) \ = \ \sum_{v \in \cM_k} N_v \left( 
\frac{1}{|S|} \sum_{z \in S} G_{M_{a,v}}(z) \right) \ .
\end{equation}

\begin{remark}
The adelic height function attached to the usual 
Mandelbrot set appeared previously in \cite{BakerHsia} and
\cite{FRL2}.  
\end{remark}

\medskip
\subsection{The function field setting}  \label{function field}
For later use, we recall a result of Benedetto and note its relevant consequences.
Let $k$ be an {\em abstract function field}, by which we mean a product formula field for which all $v \in \cM_k$ are non-archimedean.
A polynomial $\varphi \in k[T]$ is called {\em trivial over $k$} if it is
conjugate (by an invertible linear map $T \mapsto \alpha T+ \beta$ defined over $k$) 
to a polynomial defined over the constant field of $k$. 

\begin{theorem}
\label{BenedettoTheorem} \cite{BenedettoFFHeights}
Let $k$ be an abstract function field.
If $\varphi \in k[T]$ is not trivial over $k$, then 
$a \in \kbar$ is preperiodic for $\varphi$ 
if and only if $a$ belongs to the $v$-adic filled Julia set of $\varphi$ 
(i.e., $a$ stays $v$-adically bounded under iteration of $\varphi$)
for all $v \in \cM_k$.
\end{theorem}

\begin{remark}
If $k$ is a number field, then it is well known and follows easy from Northcott's theorem
that $a \in \kbar$ is preperiodic for $\varphi$ if and only if $a$ belongs to the $v$-adic filled Julia set of $\varphi$ 
for all $v \in \cM_k$.
But if $k$ is an abstract function field and $\varphi \in k[T]$ is trivial over $k$, then it
is easy to see that the conclusion of Theorem~\ref{BenedettoTheorem} fails, 
since every element of the constant field $k_0$ of $k$ stays
$v$-adically bounded for all $v$ but not every element of $k_0$ is preperiodic.
\end{remark}

\begin{cor}
\label{BenedettoCor}
Let $k$ be an abstract function field of characteristic zero such that every field $k_v$ for $v \in \cM_k$ also has residue characteristic zero,
and fix $a,c \in k$ with $c$ not in the constant field $k_0$ of $k$.  
Then the following are equivalent:
\begin{enumerate}
\item $a$ is preperiodic for the iteration of $f_c(z) = z^d + c$. 
\item $c$ is contained in $M_{a,v}$ for all $v \in \cM_k$.
\item $h_{\MM_a}(c) = 0$.
\end{enumerate}
\end{cor}

\begin{proof}
By the discussion at the beginning of \S\ref{section:GGMS},
we have $h_{\MM_a}(c) = 0$ if and only if $c$ is contained in $M_{a,v}$ for all $v \in \cM_k$.
We claim that $z^d+c$ is not trivial over $k$.
Assuming the claim, Benedetto's theorem implies that
$a$ is preperiodic for $f_c$ if and only if $a$ belongs
to the $v$-adic filled Julia set of $f_c$ for all $v \in \cM_k$.
The desired result follows, since by definition, $a$ belongs to
the $v$-adic filled Julia set of $f_c$ if and only if $c \in M_{a,v}$.

To prove the claim,
suppose for the sake of contradiction that $c \not\in k_0$ and 
$\alpha z+ \beta$ conjugates $z^d + c$ into a polynomial defined over $k_0$.
Then that conjugate is
\begin{equation}
\label{eq:BenedettoCalc}
\frac{1}{\alpha} (\alpha z+\beta)^d + \frac{1}{\alpha}(c-\beta) = \alpha^{d-1}z^d + \cdots + d\beta^{d-1}z
+ \frac{c + \beta^d - \beta}{\alpha} \in k_0[z].
\end{equation}
Since $c \not\in k_0$, there exists $v \in \cM_k$ such that $|c|_v > 1$.
By (\ref{eq:BenedettoCalc}) and our assumptions on $k$, we have 
$|\alpha|_v = 1$ and
$|d\beta^{d-1}|_v = |\beta|_v^{d-1} \leq 1$, hence $|\beta|_v \leq 1$.
Thus $|c + \beta^d - \beta|_v > 1$ by the ultrametric inequality, contradicting
(\ref{eq:BenedettoCalc}).
\end{proof}

\bigskip
\section{Proof of Theorem~\ref{MainTheorem}}

We our now ready to give the proof of our main theorem.  We also extract a stronger statement (Theorem \ref{MainTheoremBis}) in the case where the points $a$ and $b$ are algebraic.  

\subsection{The main theorem}
Fix two points $a$ and $b$ in the complex plane and a degree $d\geq 2$.  We aim to prove that the set of parameters $c\in\CC$ for which both $a$ and $b$ are preperiodic for $f_c(z) = z^d+c$ is infinite if and only if $a^d = b^d$.

\begin{proof}[Proof of Theorem~\ref{MainTheorem}]
First suppose that $a^d = b^d$.  Then $a$ is preperiodic for $f_c$ if and only if $b$ is preperiodic for $f_c$.  From Lemma \ref{MontelLemma}, the set of parameters $c$ for which $a$ is preperiodic is infinite.  

Now fix $a$ and $b$ in $\CC$, and assume that there is an infinite sequence $c_1,c_2,\ldots$ of distinct complex numbers such that $a$ and $b$ are both preperiodic for $z^d + c_n$ for all $n$.

\smallskip 

{\bf Case 1:} $a,b \in \Qbar$.

\smallskip

In this case, $c_n$ must be algebraic for all $n$.
Indeed, let $g_m(c) = f^{(m)}_c(a)$; this is a monic polynomial in $c$ of degree $d^{m-1}$ with coefficients in
the number field $k := \QQ(a) \subset \Qbar$.  Since $a$ is preperiodic for $f_{c_n}(z)$ ($n=1,2,\ldots$), 
there exist integers $\ell > m \geq 1$ (depending on $n$) such that $g_\ell(c_n) = g_m(c_n)$.
Thus $c_n$ is a root of the nonzero polynomial $g_\ell(z) - g_m(z) \in \Qbar[z]$, and hence $c_n \in \kbar$ for all $n$.  Further, we see that $a$ is also preperiodic for all $\Gal(\kbar/k)$-conjugates of $c_n$, and we deduce that $h_{\MM_a}(c_n)=0$ for all $n$.

Let $\delta_n$ be the discrete probability measure on $\CC$ supported equally on the $\Gal(\kbar/k)$-conjugates of $c_n$.  By Corollary~\ref{BREQUIcor2} the measures $\delta_n$ converge weakly to the probability measure $\mu_{M_a}$ (the equilibrium measure relative to $\infty$ for the set $M_a$) on $\PP^1(\CC)$.  By symmetry, the measures $\delta_n$ also converge weakly to $\mu_{M_b}$.  Thus $\mu_{M_a}=\mu_{M_b}$.  By Theorem~\ref{Recovery}, we conclude that $a^d = b^d$.

\smallskip

{\bf Case 2:} $a$ is transcendental.

\smallskip

In this case, $b$ is also transcendental, as otherwise each $c_n$ would be algebraic, contradicting the transcendence of $a$.  In fact, the values $a$, $b$, and $c_n$ for all $n$ are defined over the algebraic closure $\kbar$ of $k = \Qbar(a)$ in $\CC$.  
Indeed, for each $n$ there exist integers $\ell > m \geq 1$ such that 
$c_n$ is a root of the nonzero polynomial $g_\ell(z) - g_m(z) \in k[z]$,
and hence $c_n \in \kbar$ for all $n$.  
Moreover (setting $c = c_n$ for any $n$), there exist $\ell > m \geq 1$ such that $b$ is a 
root of the nonzero polynomial
$f_c^{(\ell)}(z) - f_c^{(m)}(z) \in \kbar[z]$, and hence $b \in \kbar$ as well.

Since $a$ is transcendental, the field $k= \Qbar(a)$ is isomorphic to the field 
$\Qbar(T)$ of rational functions over $\Qbar$, 
and in particular $k$ can be viewed as a product formula field with $\Qbar$ as its field of constants (cf. Remark~\ref{fgfieldremark}).
Since $a$ is preperiodic for $f_{c_n}(z)$, we have $h_{\MM_a}(c_n) = 0$ for all $n$.
Fix a place $v \in \cM_k$, let $\CC_v$ be the completion of an algebraic closure of the $v$-adic completion $k_v$,
and identify $\kbar$ with a subfield of $\CC_v$.
Let $T_m$ be the set of $\Gal(\kbar/k)$-conjugates of $c_m \in \kbar$, and define 
\[
S_n = \bigcup_{m=1}^n T_m .
\]
Then $S_n$ is a $\Gal(\kbar/k)$-stable subset of $\kbar$, $h_{\MM_a}(c)=0$ for every $c \in S_n$,
and $|S_n| \to \infty$ as $n \to \infty$.
Let $\delta_n$ be the discrete probability measure on the Berkovich projective line $\PP^1_{\Berk,v}$ over $\CC_v$ 
supported equally on the elements of $S_n$.
Let $M_{a,v} \subset \PP^1_{\Berk,v}$ be the $v$-adic generalized Mandelbrot set corresponding to $a$ 
(cf. (\ref{BerkovichMdef})).
By Theorem~\ref{AdelicHeightEquiTheorem},
the sequence $\delta_n$ converges weakly on $\PP^1_{\Berk,v}$ to the equilibrium measure $\mu_{M_{a,v}}$
for $M_{a,v}$ relative to $\infty$.
Moreover, by Proposition~\ref{capgreenprop}, the support of $\mu_{M_{a,v}}$ is equal to $\partial M_{a,v}$.

Applying the same reasoning to $b$, it follows by symmetry that $M_{a,v} = M_{b,v}$ for all places $v$ of $k$.
Hence, by Corollary~\ref{BenedettoCor}, for each fixed $c \in \kbar$, $a$ is preperiodic for $z^d+c$ if and only if $b$ is preperiodic.
Recall from the discussion at the beginning of Case 2 that if 
$c \in \CC$ and $a$ is preperiodic for $z^d + c$, then $c \in \kbar$.
It follows that for every complex number $c$, $a$ is preperiodic for $z^d + c$ if and only if $b$ is.  Theorem \ref{Recovery} then implies that $a^d = b^d$, completing the proof of the theorem.
\end{proof}

\subsection{Height bounds}
In the case where $a,b \in \Qbar$, the proof of Theorem~\ref{MainTheorem} yields the following stronger result:

\begin{theorem}
\label{MainTheoremBis}
Let $a,b \in \Qbar$ with $a^d \neq b^d$.
Then there is a real number $\varepsilon > 0$ such that 
$h_{\MM_a}(c) + h_{\MM_b}(c) \geq \varepsilon$ for all
but finitely many $c \in \Qbar$.
\end{theorem}

\begin{proof}
If we assume that the result is false, then there exists
a sequence $c_n$ of algebraic numbers such 
$\lim_{n \to \infty} h_{\MM_a}(c_n)= \lim_{n \to \infty} 
h_{\MM_b}(c_n) = 0$.
The proof now follows from the exact same argument as the proof of Case 1 of Theorem~\ref{MainTheorem}.
\end{proof}

\subsection{Remark:  non-arithmetic equidistribution}
Fix a point $a\in\CC$ and let $f_c(z) = z^d + c$.  It follows from the arguments in \cite{Dujardin:Favre} (specifically the proof of their Theorem 1) that the solutions $c\in\CC$ to equations 
	$$f_c^n(a) = f_c^{k(n)}(a),$$
weighted by multiplicity, will equidistribute to the equilibrium measure $\mu_{M_a}$ on $\partial M_a$ as $n\to\infty$ for any sequence $0 \leq k(n) < n$.  Thus, a sequence of {\em complete} sets of solutions to these equations will determine the measure $\mu_{M_a}$.  This would give an alternate argument in the proof of Theorem \ref{MainTheorem} under a much stronger hypothesis, treating the algebraic and transcendental cases simultaneously.  

However, it should be noted that under this stronger hypothesis, we are able to determine the measure $\mu_{M_a}$ and the value of $a^d$ by more classical means, without need for equidistribution.  By a straightforward argument using Montel's theorem (as in the proof of Theorem \ref{Recovery}), we find that the parameters $\{c: f_c^n(a) = f_c^{k(n)}(a)\}$ must accumulate everywhere on $\partial M_a$ as $n\to\infty$.  These parameters determine $M_a$ as a set, from which the measure $\mu_a$ is determined by classical potential theory.  The value of $a^d$ is then recovered from Theorem \ref{Recovery}.

\bigskip
\section{A variant of Theorem~\ref{MainTheorem}}

For a rational function $\varphi\in\CC(z)$, let ${\rm Preper}(\varphi)$ denote its set of preperiodic points in the Riemann sphere $\hat{\CC}$.  Our goal in this section is to prove Theorem~\ref{DynamicalTheorem}, whose statement we recall:

\begin{theorem*}
Let $\varphi,\psi \in \CC(z)$ be rational functions of degrees at least 2.
Then ${\rm Preper}(\varphi) \cap {\rm Preper}(\psi)$ is infinite if and only if
${\rm Preper}(\varphi) = {\rm Preper}(\psi)$.
\end{theorem*}

\noindent
We conclude this section with the proof of Corollary \ref{DynamicalCor}, and we state two further consequences of the proof of Theorem \ref{DynamicalTheorem} (Theorems \ref{RefinedDynamicalTheorem} and \ref{DynamicalTheoremBis}).

\subsection{Equidistribution}
The main references for this section are \cite{BREQUI} and \cite[Chapter 10]{BRBook}; see also \cite{FRL2}.
Let $k$ be a product formula field, and let $\varphi \in k(T)$ be
a rational function of degree $d \geq 2$.
Associated to $\varphi$ is the {\em Call-Silverman canonical height
function} $\hhat_{\varphi} : \PP^1(\kbar) \to \RR_{\geq 0}$.
If $k$ is a number field, then a point $P \in
\PP^1(\kbar)$ is preperiodic for $\varphi$ if and only if 
$\hhat_{\varphi}(P) = 0$.  In general, things are a little more subtle;
see Theorem~\ref{BakerIsotrivialityTheorem} below.

\medskip

For each $v \in \cM_k$, the $v$-adic {\em Arakelov-Green function} of $\varphi$ is a function
 $$g_{\varphi,v}: \PP^1_{\Berk,v} \times \PP^1_{\Berk,v} \to \RR \cup \{ +\infty \}$$ 
that takes the value $+\infty$ on the intersection of the diagonal with $\PP^1(\CC_v)$ and is finite-valued
elsewhere.  There is also a canonical probability measure $\mu_{\varphi,v}$ on $\PP^1_{\Berk,v}$; when $v$ is archimedean $\mu_{\varphi, v}$ is the measure of maximal entropy for $\varphi$ on $\PP^1(\CC)$ studied in \cite{Lyubich:entropy}, \cite{FLM}.  Each of $\mu_{\varphi,v}$ and $g_{\varphi,v}(x,y)$ determines the other by the equation
\begin{equation}  \label{mudeterminesgreen}
	\Delta_x \, g_{\varphi,v}(x,y) =  \mu_{\varphi,v} - \delta_y
\end{equation}
for every fixed $y \in \PP^1_{\Berk,v}$, where $g$ is normalized so that 
  $$\iint_{\PP^1_{\Berk,v} \times \PP^1_{\Berk,v}}
			g_{\varphi,v}(x,y) \, d\mu_{\varphi,v}(x) \, d\mu_{\varphi,v}(y) = 0.$$
(The Laplacian in (\ref{mudeterminesgreen}) is the negative of the one studied in \cite{BRBook}.)
Moreover, for all $x,y \in \PP^1(\kbar)$ with $x \neq y$ we have
\begin{equation}
\label{greendeterminesh}
\hhat_{\varphi}(x) + \hhat_{\varphi}(y) = 
\sum_{v \in \cM_k} N_v \, g_{\varphi,v}(x,y).
\end{equation}

\medskip

If $v$ is archimedean, it is well known that $\supp(\mu_{\varphi,v})$ is
equal to the complex Julia set of $\varphi$.
For $v$ non-archimedean, the Berkovich Julia set of $\varphi$ is
{\em defined} in \cite[Chapter 10]{BRBook} to be the support of 
$\mu_{\varphi,v}$.  This turns out to be equivalent to several other,
more topological, characterizations of the Berkovich Julia set.

\medskip

If $S$ is a finite subset of $\PP^1(\ksep)$ which is stable under
$\Gal(\ksep/k)$, we define 
\[
\hhat_{\varphi}(S) = \frac{1}{|S|} \sum_{P \in S} \hhat_{\varphi}(P).
\]
The following equidistribution theorem was proved independently by
Baker--Rumely \cite{BREQUI}, Chambert-Loir  \cite{CL}, and Favre--Rivera-Letelier \cite{FRL2} in the
number field case.  For the present formulation in terms of an
arbitrary product formula field, see \cite[Theorem 10.24]{BRBook}.

\begin{theorem}
\label{RationalEquiThm}
Let $k$ be a product formula field, 
and let $\varphi \in k(T)$ be a rational function of degree $d \geq 2$.
Suppose $S_n$ is a sequence of $\Gal(\ksep/k)$-invariant finite subsets of $\ksep$ with $|S_n| \to \infty$
and $\hhat_{\varphi}(S_n) \to 0$. 
Fix $v \in \cM_k$, and for each $n$ let $\delta_n$ be the discrete probability measure on
$\PP^1_{\Berk,v}$ supported equally on the elements of $S_n$.
Then the sequence of measures $\{\delta_n\}$ converges weakly
to $\mu_{\varphi,v}$ on $\PP^1_{\Berk,v}$.
\end{theorem}

%

\subsection{Isotriviality}
Let $k$ be an abstract function field; see \S\ref{function field} for the definition.  
A map $\varphi \in k(T)$ of degree $d$ 
is said to have {\em good reduction} at 
$v \in \cM_k$ if $\varphi \ = \ f/g$ for some $f,g \in \Ocal_v[T]$ 
whose reductions $\overline{f},\overline{g}$
are polynomials in $\tilde{k}_v[T]$ with no common roots in any
algebraic closure of $\tilde{k}_v$ and such that at
least one of $\overline{f},\overline{g}$ has degree $d$.
(Here $\Ocal_v$ denotes the valuation ring of $k_v$, and $\tilde{k}_v$
its residue field.)


The following result is proved in \cite{BakerFFH}; the equivalence of
(2) and (3) is due originally to Rivera-Letelier and another proof of the equivalence of (1) and (2) can be found in \cite{PetscheSzpiroTucker}.

\begin{proposition}
\label{BakerIsotrivialityProp}
Let $k$ be an abstract function field, and let $\varphi \in k(T)$ be a rational map of degree at least 2.
Then the following are equivalent:
\begin{enumerate}
\item $\varphi$ is defined over the constant field of $k$
\item $\varphi$ has good reduction at every $v \in \cM_k$
\item The canonical measure $\mu_{\varphi,v}$ is a point mass supported at the Gauss point of $\PP^1_{\Berk,v}$ for all 
$v \in \cM_k$.
\end{enumerate}
\end{proposition}

Assume now that the constant field of
$k$ is algebraically closed.
A rational map $\varphi \in k(T)$ is called {\em isotrivial over $k$} if it is
conjugate, by a  linear fractional transformation defined over $\kbar$, 
to a rational map defined over the constant field of $k$. 
The following result is also proved in \cite{BakerFFH}:

\begin{theorem}
\label{BakerIsotrivialityTheorem}
Let $k$ be an abstract function field with an algebraically closed
field of constants.
If $\varphi \in k(T)$ is a rational map of degree $d \geq 2$
which is not isotrivial, then a point $P \in \PP^1(\kbar)$ satisfies $\hhat_{\varphi}(P)=0$
if and only if $P$ is preperiodic for $\varphi$.
\end{theorem}

Modulo the assumption that the constant field of $k$ is algebraically closed,
Theorem~\ref{BakerIsotrivialityTheorem} is a generalization of Theorem~\ref{BenedettoTheorem}.
See \cite{ChatzHrush1,ChatzHrush2} for a higher-dimensional generalization of 
Theorem~\ref{BakerIsotrivialityTheorem} (which also does not require the assumption that the constant field
of $k$ is algebraically closed).

\subsection{Preperiodic points for rational functions}
We can now give the proof of Theorem~\ref{DynamicalTheorem}.

\begin{proof}[Proof of Theorem~\ref{DynamicalTheorem}]
First note that the set of preperiodic points of a rational function of degree $\geq 2$ is always infinite; see \cite[\S{6.2}]{BeardonBook}.  
Therefore, the equality ${\rm Preper}(\varphi) = {\rm Preper}(\psi)$ clearly implies that the intersection is infinite.  It remains to prove the reverse implication.  

Suppose there is an infinite sequence $a_n$ of complex numbers which are preperiodic for both $\varphi$ and $\psi$.

\smallskip

{\bf Case 1:} $\varphi,\psi \in \Qbar(T)$.

\smallskip

In this case, $a_n$ is algebraic for all $n$.  Let $k$ be a number field over which both $\varphi$ and $\psi$
are defined.  Fix a place $v$ of $k$ and let $\delta_n$ be the discrete probability measure on $\PP^1(k_v)$ supported equally on the $\Gal(\kbar/k)$-conjugates of $a_n$.  By Theorem~\ref{RationalEquiThm}, the sequence $\delta_n$ converges weakly on $\PP^1_{Berk,v}$ to both $\mu_{\varphi,v}$ and $\mu_{\psi,v}$, hence $\mu_{\varphi,v} = \mu_{\psi,v}$.  By (\ref{mudeterminesgreen}), we therefore have 
$g_{\mu_{\varphi},v}(x,y) = g_{\mu_{\psi},v}(x,y)$
for all $v \in \cM_k$ and all $x,y \in \PP^1_{\Berk}$.
From the identity (\ref{greendeterminesh}),
it follows (letting $y = a_n$ for some $n$, so that $\hhat_{\varphi}(y)
= \hhat_{\psi}(y) = 0$) that $\hhat_{\varphi}(x) = \hhat_{\psi}(x)$
for all $x \in \PP^1(\kbar)$.  As all preperiodic points lie in $\kbar$, we may conclude that ${\rm Preper}(\varphi) = {\rm Preper}(\psi)$.

\smallskip

{\bf Case 2:} Neither $\varphi$ nor $\psi$ is conjugate to a rational map defined over $\Qbar$.

\smallskip

In this case, all $a_n$'s are defined over $\kbar$, where $k$ is the finitely generated field extension of $\Qbar$
generated by the coefficients of $\varphi$ and $\psi$.
Equip $k$ with a product formula structure as in Remark~\ref{fgfieldremark}.
Let $T_m$ be the set of $\Gal(\kbar/k)$-conjugates of $a_m$, and set 
\begin{equation}
\label{eq:Sn2}
S_n = \bigcup_{m=1}^n T_m .
\end{equation}
For each $v \in \cM_k$, let $\delta_n$ be the discrete probability measure on $\PP^1_{\Berk,v}$ supported equally on the 
elements of $S_n$.
By Theorem~\ref{RationalEquiThm}, the sequence $\delta_n$ converges weakly on $\PP^1_{\Berk,v}$ to both 
$\mu_{\varphi,v}$ and $\mu_{\psi,v}$, hence 
$\mu_{\varphi,v} = \mu_{\psi,v}$ for all $v \in \cM_k$.
From equations (\ref{mudeterminesgreen}) and (\ref{greendeterminesh}), we deduce that the height functions $\hat{h}_{\varphi}$ and $\hat{h}_{\psi}$ must coincide on $\PP^1(\bar{k})$.

By assumption, neither $\varphi$ nor $\psi$ is conjugate to a map 
defined over $\Qbar$.  
By Theorem~\ref{BakerIsotrivialityTheorem},
a point $x \in \PP^1(\kbar)$ is preperiodic for $\varphi$ (resp. 
$\psi$) if and only if $\hhat_{\varphi}(x) = 0$ (resp. $\hhat_{\psi}(x)=0$).
We conclude that $\varphi$ and $\psi$ have the same set of
preperiodic points in $\kbar$, and therefore the same set of 
preperiodic points in $\CC$.

\smallskip

{\bf Case 3:} $\varphi$ is conjugate to a rational map defined over 
$\Qbar$.

\smallskip

Replacing $\varphi$ by a conjugate, we may assume
without loss of generality that $\varphi$ is defined over $\Qbar$.
We claim that (in these new coordinates) $\psi$ is defined over 
$\Qbar$ as well, so that we are back in Case 1.

As in Case 2, all $a_n$'s are defined over $\kbar$, where $k$ is the finitely generated field extension of $\Qbar$
generated by the coefficients of $\varphi$ and $\psi$.
We may assume that $k/\QQ$ is transcendental, since otherwise we're done.  So as in Case 2, $k$ can be endowed with a product formula structure with respect to which all places $v \in \cM_k$ are non-archimedean and the constant field of $k$ is $\Qbar$.
For each $v \in \cM_k$, let $\delta_n$ be the discrete probability measure on $\PP^1_{\Berk,v}$ supported equally on the 
elements of $S_n$, where $S_n$ is defined as in (\ref{eq:Sn2}).
By Theorem~\ref{RationalEquiThm}, the sequence $\delta_n$ converges weakly on $\PP^1_{\Berk,v}$ to both 
$\mu_{\varphi,v}$ and $\mu_{\psi,v}$, hence 
$\mu_{\varphi,v} = \mu_{\psi,v}$ for all $v \in \cM_k$.

Since $\varphi$ is defined over the constant field $\Qbar$ of $k$, 
$\varphi$ has good reduction at every $v \in \cM_k$.
Equivalently, $\mu_{\varphi,v}$ is a point mass supported at
the Gauss point of $\PP^1_{\Berk,v}$ for all $v \in \cM_k$.
As $\mu_{\varphi,v} = \mu_{\psi,v}$ for all $v \in \cM_k$, we deduce
from Proposition~\ref{BakerIsotrivialityProp}
that $\psi$ has good reduction at every $v \in \cM_k$
and hence is defined over the constant field $\Qbar$ of $k$,
as claimed.
\end{proof}


\subsection{Shared Julia sets and further consequences}
Corollary \ref{DynamicalCor} states that if rational functions $\varphi$ and $\psi$ have infinitely many preperiodic points in common, then they must have the same Julia set.

\begin{proof}[Proof of Corollary \ref{DynamicalCor}]
Assume that $\varphi$ and $\psi$ each have degree $>1$ and that they have infinitely many preperiodic points in common.  From Theorem \ref{DynamicalTheorem}, they share all of their preperiodic points.  

It is well known that repelling periodic points are dense in the Julia set $J(\varphi)$ (see \cite[Theorem 14.1]{MilnorDynamicsBook}), so clearly $J(\varphi) \subset \overline{{\rm Preper}(\varphi)}$.  
Furthermore, the Julia set has no isolated points by \cite[Corollary 4.11]{MilnorDynamicsBook}, so every point of $J(\varphi)$ is an accumulation point of ${\rm Preper}(\varphi)$.  
On the other hand, preperiodic points form a discrete subset of the Fatou set (see \cite[Proposition 2]{PetscheSIntegral} for a proof).
We conclude that the Julia set of a rational map $\varphi$ of degree at least $2$ is equal to the set of accumulation points of ${\rm Preper}(\varphi)$.  
From the equality ${\rm Preper}(\varphi) = {\rm Preper}(\psi)$, we may therefore deduce that $J(\varphi) = J(\psi)$.  
\end{proof}

Much is known about rational functions in one complex variable which share the same Julia set; see for example \cite{Beardon} (for the polynomial case) and \cite{LevinPrzy}.  For example, if $\varphi(z) = z^2 + c$ and $\psi(z) = z^2 + c'$ with $c \neq c'$, then $\varphi$ and $\psi$ have distinct Julia sets; see \cite[\S4]{Beardon}.  For polynomials defined over a number field\footnote{See \cite[\S{3}]{GTZhang} for a recent generalization to polynomials with arbitrary complex coefficients.}, we use \cite[Corollary 25]{KawSilv} (which is itself based on the classification in \cite{Beardon}) to strengthen the conclusion of Theorem \ref{DynamicalTheorem}: 

\begin{theorem}
\label{RefinedDynamicalTheorem}
Let $\varphi, \psi \in\Qbar[z]$ be polynomials of degrees at least 2 with algebraic coefficients, and fix an embedding of $\Qbar$ into $\CC$.  Then the set of points which are preperiodic for both $\varphi$ and 
$\psi$ is infinite if and only if there are positive integers $m,n$ and a linear fractional transformation $\lambda$ defined over $\Qbar$ such that $\tilde{\varphi} := \lambda^{-1}\varphi \lambda$ and $\tilde{\psi} := \lambda^{-1}\psi \lambda$ satisfy one of the following:
\begin{enumerate}
\item $\tilde{\varphi}(z) = z^n$ and $\tilde{\psi}(z) = \eta z^m$ for some root of unity $\eta$.  In this case $J(\tilde{\varphi})=J(\tilde{\psi})$ is the complex unit circle.
\item $\tilde{\varphi}(z) = \pm T_n(z)$ and $\tilde{\psi}(z) = \pm T_m(z)$, where $T_n(z)$ is the $n^{\rm th}$ Tchebycheff polynomial.  In this case
$J(\tilde{\varphi})=J(\tilde{\psi}) = [-2,2]$.
\item $J(\tilde{\varphi}) = J(\tilde{\psi})$ has a Euclidean symmetry group which is cyclic of order $k$ (consisting of rotations about the origin in $\CC$) and there are $k^{\rm th}$ roots of unity $\eta_1$ and $\eta_2$ and a polynomial $\gamma \in \Qbar[z]$ such that 
$\eta_1 \tilde{\varphi}(z)$ and $\eta_2 \tilde{\psi}(z)$
are both iterates of $\gamma$.
\end{enumerate}  
\end{theorem}

\begin{proof}
If $\varphi,\psi$ are defined over a number field $k$ and $\varphi,\psi$ have infinitely many preperiodic points in common, then the proof of Theorem~\ref{DynamicalTheorem} shows that 
$\mu_{\varphi,v} = \mu_{\psi,v}$ for all places $v$ of $k$ and hence 
that $\hhat_{\varphi} = \hhat_{\psi}$.
Combined with \cite[Corollary 25]{KawSilv}, this proves
Theorem~\ref{RefinedDynamicalTheorem}.
\end{proof}

\medskip

\begin{remark}
If $\varphi,\psi \in \CC(T)$ are Latt{\`e}s maps coming from multiplication by $2$ on the elliptic curves $E,E'$,
respectively, then $\varphi$ and $\psi$ have infinitely many preperiodic points in common if and only if $E$ and $E'$
are isogenous.  This follows from \cite[Proof of Lemma 3.7]{BogomTschink}.
\end{remark}

\medskip

If we replace the sequence $a_n$ of preperiodic points
in the proof of Theorem~\ref{DynamicalTheorem} with a sequence
of points such that $\lim_{n \to \infty} \hhat_{\varphi}(a_n)= 
\lim_{n \to \infty} \hhat_{\psi}(a_n)=0$, these observations also prove the following result:

\begin{theorem}
\label{DynamicalTheoremBis}
Let $\varphi,\psi \in \Qbar(T)$ be rational functions of degree at least 2, and assume that the canonical height functions 
$\hhat_\varphi$ and $\hhat_\psi$ on $\PP^1(\Qbar)$ are distinct.
(This will be true, in particular, if the complex Julia sets of $\varphi$
and $\psi$ are distinct for some embedding of $\Qbar$ into $\CC$.)
Then there is a real number $\varepsilon > 0$ such that 
$\hhat_{\varphi}(a) + \hhat_{\psi}(a) \geq \varepsilon$ for all
but finitely many $a \in \Qbar$.
\end{theorem}

\bibliographystyle{alpha}
\bibliography{BDM}


\end{document}